\documentclass[oneside,a4paper,reqno,11pt]{amsart}
\usepackage{latexsym,amsfonts,amssymb,amsmath,amscd}
\usepackage{graphicx,color,ifpdf}
\usepackage{enumitem}
\usepackage[all]{xy}
\allowdisplaybreaks
\newtheorem*{thma}{Theorem A}
\newtheorem*{thmb}{Theorem B}
\usepackage{aliascnt}
\usepackage[bookmarks=true,pdfstartview=FitH, pdfborder={0 0 0},
colorlinks=true,citecolor=blue, linkcolor=blue]{hyperref}

\newtheorem{theorem}{Theorem}[section]
\newaliascnt{lemma}{theorem}
\newtheorem{lemma}[lemma]{Lemma}
\aliascntresetthe{lemma}

\newaliascnt{proposition}{theorem}

\aliascntresetthe{proposition}

\newaliascnt{corollary}{theorem}

\aliascntresetthe{corollary}

\theoremstyle{definition}
\newaliascnt{definition}{theorem}

\aliascntresetthe{definition}

\newaliascnt{example}{theorem}

\aliascntresetthe{example}

\newaliascnt{exercise}{theorem}

\aliascntresetthe{exercise}

\newaliascnt{question}{theorem}

\aliascntresetthe{question}

\newaliascnt{problem}{theorem}

\aliascntresetthe{problem}

\theoremstyle{remark}
\newaliascnt{remark}{theorem}

\aliascntresetthe{remark}

\newaliascnt{notation}{theorem}

\aliascntresetthe{notation}

\newaliascnt{fact}{theorem}

\aliascntresetthe{fact}

\numberwithin{equation}{theorem}%
\numberwithin{figure}{theorem}
\newcommand{\Cal}[1]{\mathcal{#1}}
\renewcommand{\Bbb}[1]{\mathbb{#1}}
\newcommand{\injrad}{\operatorname{injrad}}
\newcommand{\conj}{\operatorname{conj}}
\newcommand{\econj}{\operatorname{conj}^\varepsilon}

\begin{document}
\title[Conjugate point and geodesic loop]{On conjugate points and 
geodesic loops in a complete Riemannian manifold}
\author{Shicheng Xu}
\date{\today}
\address{School of Mathematics, Capital Normal University,
Beijing, China}
\email{shichengxu@gmail.com}
%\fntext[fn1]{Project 11171143 supported by NSFC.}
\thanks{Project 11171143 supported by National Natural Science 
Foundation of China}
\begin{abstract}
A well-known Lemma in Riemannian geometry by Klingenberg says 
that if $x_0$ is a minimum point of the distance function 
$d(p,\cdot)$ to $p$ in the cut locus $C_p$ of $p$,
then either there is a minimal geodesic from $p$ to $x_0$ along which 
they are conjugate, or there is a geodesic loop at $p$ that smoothly 
goes through $x_0$. In this paper, we prove that: for any point $q$ 
and any local minimum point $x_0$ of 
$F_q(\cdot)=d(p,\cdot)+d(q,\cdot)$ in $C_p$, 
either
$x_0$ is conjugate to $p$ along each minimal geodesic connecting 
them, or there is a geodesic from $p$ to $q$ passing through $x_0$. 
In particular, for any local minimum point $x_0$ of 
$d(p,\cdot)$ in $C_p$, either $p$ and $x_0$ are conjugate along 
every minimal geodesic from $p$ to $x_0$, or there is a geodesic loop 
at $p$ that smoothly goes through $x_0$. Earlier results based on 
injective radius estimate would hold under weaker conditions.
\end{abstract}

\maketitle
\section{Introduction}

Let $M$ be a complete Riemannian manifold. For a point $p\in M$, let 
$T_pM$ be the tangent space at $p$ and $\exp_p:T_pM\to M$ be the 
exponential map. For any unit 
vector $v\in T_pM$, let $\sigma(v)$ be the supremum of $l$ such that
the geodesic $\exp tv:[0,l]\to M$ is minimizing, $\kappa(v)$ be the 
supremum of $s$ such that there is no conjugate point of $p$ along 
$\exp_p tv:[0,s)\to M$.
Let $$\tilde C_p=\{\sigma(v)v:
\text{ for all unit vector $v\in T_pM$}\}$$ be the {\it tangential 
cut locus} of $p$, 
$$\tilde J_p=\{\kappa(v)v: \text{ for all unit vector $v\in T_pM$}\}$$
be the {\it tangential conjugate locus},
and $C_p=\exp \tilde C_p\subset M$ be the {\it cut locus} of $p$. Let 
$$\tilde D_p=\{tv\,|\, 0\le t<\sigma(v), 
\text{ for all unit vector }v\in T_pM\}$$
be the maximal open domain of the origin in $T_pM$ such that the 
restriction 
$\exp_p|_{\tilde D_p}$ of $\exp_p$ on $\tilde D_p$ is injective. 
Any geodesic throughout the paper is assumed to be parametrized by 
arclength. A geodesic $\gamma:[0,l]\to M$ is called a {\it 
geodesic loop at $p$} if $\gamma(0)=\gamma(l)=p$.
Let $d(\cdot,\cdot)$ be the Riemannian distance function on $M$.
In \cite{Klingenberg1959sphere} Klingenberg proved a lemma that is
well-known now in Riemannian geometry and particularly useful
in injectivity radius estimate.

\begin{lemma}[Klingenberg \cite{Klingenberg1959sphere, 
Klingenberg1995riemannian}]\label{ lem:OldKL}
If $x_0\in C_p$ satisfies that $d(p,x_0)=d(p,C_p)$,
then either there is a minimal geodesic from $p$ to $x_0$ along which 
they are conjugate, or there exist exactly two minimal geodesics from 
$p$ to $x_0$ that form a geodesic loop at $p$ smoothly passing 
through $x_0$.
\end{lemma}

Recently this lemma was generalized to the case of two 
points (\cite{InnamiShiohamaSoga2012pole}). Let $p,q$ be two points 
in a complete 
Riemannian manifold $M$ such that $C_p\neq \emptyset$ and $q\not\in 
C_p$. Let
$F_{p;q}:C_p\to \Bbb R$ be a function defined on $C_p$ by
$F_{p;q}(x)=d(p,x)+d(x,q)$. 
Innami, Shiohama and Soga proved in \cite{InnamiShiohamaSoga2012pole} 
that

\begin{lemma}[\cite{InnamiShiohamaSoga2012pole}] 
\label{lem:OldGeneralization}
If $\tilde J_p\cap \tilde C_p=\emptyset$, then for any minimum 
point $x_0\in C_p$ of $F_{p;q}(x)=d(p,x)+d(x,q)$, 
there exist a geodesic (and at most two) $\alpha:[0,F_{p;q}(x_0)]\to 
M$ from $p$ to $q$ such that $\alpha(d(p,x_0))=x_0$.
\end{lemma}

Note that if $p=q$, then Lemma~\ref{lem:OldGeneralization} is 
reduced to the case of Klingenberg's Lemma. In this paper, we improve
both results in the above to the following theorems
whose constraints are sharp in general.

\begin{thma}[Generalized Klingenberg's Lemma]
\label{thma:NewGeneralization}
Let $M$ be a complete Riemannian manifold and $p,q\in M$ such that
$C_p\neq \emptyset$ and $q\not\in C_p$.
Let $x_0\in C_p$ such that $F_{p;q}(x_0)=d(p,x_0)+d(q,x_0)$ is local 
minimum of $F_{p;q}$ in $C_p$. 
Then either $p$ and $x_0$ are conjugate along every minimal geodesic 
connecting them, or there is 
a geodesic (and at most two) $\alpha:[0,F_{p;q}(x_0)]\to 
M$  from $p$ to $q$ such that $\alpha(d(p,x_0))=x_0$.
\end{thma}

\begin{thmb}[Improved Klingenberg's Lemma]
\label{thmb:ImprovedKL}
Let $M$ be a Riemannian manifold and $p$ be a point in $M$. Let 
$x_0\in C_p$ such that $d(p,x_0)$ is a local minimum of $d(p,\cdot)$  
in $C_p$. If there is a minimal geodesic from $p$ to $x_0$ along 
which $p$ is not conjugate to $x_0$, then there are exactly two 
minimal geodesics from $p$ to $x_0$ that form a whole geodesic 
smoothly passing through $x_0$. 

Moreover, if $d(p,x_0)$ is also a 
local minimum of $d(x_0,\cdot)$ in $C_{x_0}$, then the two minimal 
geodesics form a closed geodesic.
\end{thmb}

Motivated by Theorem~B, we will call a cut 
point $q$ of $p$ an {\it essential conjugate} 
point 
of $p$ if $p$ is conjugate to $q$ along every minimal geodesic 
connecting them. Geodesic loops may not exist if there are essential 
conjugate points. For example, a complete noncompact Riemannian 
manifold of positive sectional curvature always contains a {\it 
simple} point $p$ (i.e. there is no geodesic loop at $p$) whose cut 
locus is nonempty \cite{GromollMeyer1969positive}. By Theorem~B, 
every local minimum point of $d(p,\cdot)$ in $C_p$ must be 
essentially conjugate to $p$. 

The conjugate radius had been involved in the injectivity radius 
estimate besides the length of geodesic loops 
(see \cite[Lemma~4]{Klingenberg1959sphere} or 
\cite[Lemma~1.8]{Abresch1997injectivity}).
Recall that the {\it conjugate radius} at $p$ is defined by
$$\conj (p)=\min\{\kappa(v)\,|\, 
\text{ for all unit vector $v\in T_pM$}\}$$
and the {\it conjugate radius}
of $M$, $\conj(M)=\inf_{p\in M}\conj(p).$
The {\it injectivity radius} of $p$ is defined by 
$$\injrad(p)=\min\{\sigma(v)\,|\, 
\text{ for all unit vector $v\in T_pM$}\} $$
and {\it injectivity radius} of $M$, $\injrad(M)=\inf_{p\in 
M}\injrad(p)$.
It follows from Theorem~B that 
the conjugate radius in the injectivity radius 
estimate can be replaced by the distance from $p$ to its
essential conjugate points.
Let $J^\varepsilon_p=C_p\setminus\exp_p(\tilde C_p\setminus \tilde 
J_p))$ be the the set consisting of all essential conjugate points of 
$p$. Let $$\econj(p)=\begin{cases}
d(p,J^\varepsilon_p) & \text{if } J^\varepsilon_p\neq\emptyset,\\
+\infty &\text{otherwise}.
\end{cases}$$ and
$\econj(M)=\min_{p\in M} \econj(p)$.
Then the injectivity radius can be expressed in the following way, 
where non-essential conjugate points are covered by geodesic loops.
\begin{theorem}\label{thm:injrad}
\begin{align*}
&\injrad(p)
=\min\left\{\begin{array}{l}
\econj(p), \\
\text{half length of the shortest geodesic loop at $p$}
\end{array}\right\}; \\
&\injrad(M)=\min\left\{\begin{array}{l}
\econj(M), \\
\text{half length of the shortest closed geodesic in $M$}
\end{array}\right\}.
\end{align*}
\end{theorem}

By Theorem~\ref{thm:injrad}, in this paper we call $\econj(p)$ the 
{\it essential conjugate radius} of $p$ and $\econj(M)$ the {\it 
essential conjugate radius} of $M$. 

There have been rich results in Riemannian geometry where an upper 
sectional curvature bound $K$ shows up to offer a lower bound 
$\frac{\pi}{\sqrt{K}}$ of the conjugate radius, which now could be 
weakened to the essential conjugate radius. 
For instance, one is able to generalize 
Cheeger's finiteness theorem of diffeomorphism classes to Riemannian 
manifolds whose essential conjugate radius has a lower bound but
sectional curvature has no upper bound. 

\begin{theorem}\label{thm:finiteness}
 For any positive integer $n$, real numbers $D,v,r>0$ and 
$k$, there are only finite $C^{\infty}$-diffeomorphism classes in 
the set consisting of $n$-dimensional Riemannian manifolds
whose sectional curvature is bounded below by $k$, diameter $\le D$, 
volume $\ge v$ and essential conjugate radius $\ge r$.
\end{theorem}

In general $\econj(p)\neq\conj(p)$. $\mathbb{R}P^n$ with the 
canonical metric is a trivial example. An immediate question is, when 
$\econj(p)=\conj(p)$?
Next two applications of Theorem~B
offers examples where $\econj(p)$ coincides with $\conj(p)$.

\begin{theorem}\label{thm:soul}
Let a point $p$ be the soul of a complete noncompact 
Riemannian manifold of nonnegative sectional curvature in sense 
of \cite{CheegerGromoll1972soul}. Then 
either the nearest point to $p$ in $C_p$ is an essential conjugate 
point, or $\exp_p:T_pM\to M$ is a diffeomorphism.
\end{theorem}
By Theorem~B, Theorem~\ref{thm:soul} directly 
follows from the fact that the soul is totally convex 
(\cite{CheegerGromoll1972soul}).
For a closed Riemannian manifold $M$, the {\it radius} of $M$ is
defined by $\operatorname{rad}(M)=\min_{p\in 
M}\max\{d(p,x)\,|\, \text{ for any $x\in M$}\}.$

\begin{theorem}\label{thm:radiussphere}
Let $M$ be a complete Riemannian manifold whose sectional curvature 
$\ge 1$ and radius $\operatorname{rad}(M)>\frac{\pi}{2}$. Then 
either there are two points $p,q\in M$ of distance
$d(p,q)=d(p,C_p)=d(q,C_q)<\operatorname{rad}(M)$ and
$p,q$ are essentially conjugate to each other,
or $M$ is isometric to a sphere of constant curvature 
$\frac{\pi^2}{\operatorname{rad}^2(M)}$.
\end{theorem}

Xia proved (\cite{Xia2006roundsphere}) that if the manifold $M$ in 
Theorem~\ref{thm:radiussphere} satisfies $\conj(M)\ge 
\operatorname{rad}(M)>\frac{\pi}{2}$, then $M$ is isometric to a 
sphere of constant curvature. Theorem~\ref{thm:radiussphere} follows 
his proof after replacing Klingenberg's lemma to Theorem~B. In 
particular, if $\operatorname{rad}(p)= \max_{x\in 
M}d(p,x)>\frac{\pi}{2}$, then $\injrad(p)=\econj(p)=\conj(p)$.
Weaker versions of Theorem~\ref{thm:radiussphere} can also be found 
in \cite{Wang2004largeradius}. 

Theorem~A provides
a new characterization on general Riemannian manifolds,
which is an improvement of Theorem~1 in 
\cite{InnamiShiohamaSoga2012pole}. 

\begin{theorem}\label{thm:newcharacter}
Let $M$ be a complete Riemannian manifold and $p$ be a point in $M$ 
such that 
$C_p\neq\emptyset$. Then
\begin{enumerate}
\item[$(\ref{thm:newcharacter}.1)$] either $p$ has an essential 
conjugate point;
 \item[$(\ref{thm:newcharacter}.2)$] or there exist at least two 
geodesics connecting $p$ 
and every point $q\in M$ (regarding the single point $p$ as a 
geodesic when $p=q$). \end{enumerate}
\end{theorem}

We now explain the idea of our proof of Theorem~A. Let $p,q\in M$ 
such that 
$C_p\neq\emptyset$, $q\not\in C_p$ and $x_0\in C_p$ is a minimum point 
of $F_{p;q}$ in $C_p$.
Let us consider the special case that 
there is a unique minimal geodesic $[qx_0]$ connecting $q$ and $x_0$.
A key observation from \cite{InnamiShiohamaSoga2012pole}
is that the level set of $\{F_{p;q}\le C\}$ is star-shaped at both 
$p$ and $q$. In particular, if two minimal geodesics $[px_0]$ and 
$[x_0q]$ from $p$ to $x_0$ and from $x_0$ to $q$ are broken at $x_0$, 
then for any point $x\neq x_0$ in $[px_0]$ and any minimal geodesic 
$[qx]$ connecting $q$ and $x$, we have $[qx]\cap C_p=\emptyset$. Thus 
$[qx]$ admit a unique lifting $\widetilde{[qx]}$ in $\tilde D_p$.
If two minimal geodesics, say $[px_0]_1$ and $[px_0]_2$, do not form 
a whole geodesic with $[x_0q]$ at the same time, 
then by moving $x$ to $x_0$ along $[px_0]_1$ 
and $[px_0]_2$ respectively, one may expect two 
liftings of $[qx_0]$ in $\tilde D_p$ with different endpoints when 
partial limits of $\widetilde{[qx]}$ exist.

Under the condition that $\tilde J_p\cap 
\tilde C_p=\emptyset$ as in 
Lemma~\ref{lem:OldGeneralization} and 
\cite{InnamiShiohamaSoga2012pole}, 
all such minimal geodesics 
$[px]$ are clearly
definite away from the tangential conjugate locus.
So we are able to take limit to meet a contradiction, for 
the lift $\widetilde{[qx_0]}$ of $[qx_0]$ is unique in the 
closure of $\tilde D_p$.

In the general case one does not know whether $[qx_0]$ admits a 
lifting at its endpoint in the closure of $\tilde D_p$, nor the 
liftings $\widetilde{[qx]}$ have a partial limit as $x$ approaches 
$x_0$ along $[px_0]$. 
We will prove that, if there is a minimal geodesic $\alpha$ from 
$p$ to $x_0$ along which they are not conjugate to each other, and 
the union of $\alpha$ and $[qx_0]$ is broken at 
$x_0$, then
$[qx_0]$ always has a lifting in the 
closure of $\tilde D_p$, which 
share a common endpoint with
$\tilde{\alpha}$ (see Lemma~\ref{lem:propertyoflift}). 

The uniqueness of the minimal geodesic $[qx_0]$ is 
not an essential problem, because by (\ref{lem:propertyoflift}.4) and 
(\ref{lem:propertyoflift}.5) we are always able to move $q$ 
along $[qx_0]$ while keeping $F_{p;q}$ minimal at $x_0$ in $C_p$. A 
local minimum of $F_{p;q}$ can be reduced 
to the minimum case similarly.

The detailed proof of Theorem~A will be 
given in the next section.

\section{Proof of the generalized Klingenberg's lemma}
In this section we will prove Theorem~A.
Let $M$ be a complete Riemannian manifold and $p,q$ be two points in 
$M$ such that $C_p\neq\emptyset$ and $q\not\in C_p$. Let us consider 
the function 
$$F_{p;q}:C_p\to \Bbb R, 
\quad F_{p;q}(x)=d(x,q)+d(x,p)$$ and assume that $F_{p;q}$ takes its 
minimum 
at $x_0\in C_p$. Because $q$ is not a cut point of $p$, $q\neq 
x_0$. Let $\gamma:[0,d(q,x_0)]\to M$ be a minimal geodesic from 
$q=\gamma(0)$ to $x_0=\gamma(d(q,x_0))$. Because $x_0$ is a cut 
point of $p$, for any $0\le t< d(q,x_0)$ we have
\begin{align*}
F_q(\gamma(t))&=d(q,\gamma(t))+d(\gamma(t),p)\\
&< d(q,\gamma(t))+d(\gamma(t),x_0)+d(x_0,p)\tag{$2.1$}\\ 
&=d(q,x_0)+d(x_0,p)\\
&=F_q(x_0)\\
&=\min F_q.
\end{align*} \addtocounter{theorem}{1}
Therefore $\gamma(t)$ $(0\le t< d(q,x_0))$ is not a cut point of 
$p$, 
and we are able to lift $\left.\gamma\right|_{[0,d(q,x_0))}$ to 
$(\exp_p|_{\tilde{D}_p})^{-1}\circ\gamma$ uniquely in the {\it 
tangential segment domain} $\tilde D_p\subset T_pM$,
where $\tilde D_p$ is the maximal open domain of the origin in 
$T_pM$ such that the restriction 
$\exp_p|_{\tilde D_p}$ of $\exp_p$ on $\tilde D_p$ is injective.

We now prove that, if $x_0\in \exp_p(\tilde{C}(p)\setminus 
\tilde{J}(p))$, then either $\gamma$ can be lifted on the 
whole interval $[0,d(q,x_0)]$ such that the endpoint of the 
lift is a regular point of $\exp_p$, or $\gamma$ can be extended to a 
geodesic from $q$ to $p$ that goes through $x_0$.

\begin{lemma}\label{lem:propertyoflift}
Assume that there is a minimal geodesic $\alpha:[0,d(p,x_0)]\to M$ 
from $p=\alpha(0)$ to $x_0=\alpha(d(p,x_0))$ along which $p$ is 
not conjugate to $x_0$. Let $w=d(p,x_0)\alpha'(0)\in T_pM$, where 
$\alpha'(0)$ is the unit tangent vector of $\alpha$ at $p$.
Then
\begin{enumerate}
 \item[$(\ref{lem:propertyoflift}.1)$] either $\gamma$ and $\alpha$ 
form a whole geodesic at $x_0$;
 \item[$(\ref{lem:propertyoflift}.2)$] or there is a unique smooth 
lift 
$\tilde\gamma:[0,d(q,x_0)]\to T_pM$ of $\gamma:[0,d(q,x_0)]\to M$ in 
the tangential segment domain $\tilde D_p\subset T_pM$ such that 
$\tilde\gamma(0)=(\exp_p|_{\tilde{D}_p})^{-1}\gamma(0)$
and $\tilde\gamma(d(q,x_0))=w$.
\end{enumerate}
\end{lemma}

\begin{proof}
We first prove the case that $\gamma$ is the unique minimal geodesic 
from $q$ to $x_0$.
Assume that $\gamma$ and $\alpha$ do not form a whole geodesic at 
$x_0$. Let $\{\alpha(s_i)\}$ $(0<s_i<d(p,x_0))$ be a sequence 
of interior points of $\alpha$ that converges to $x_0$ as 
$i\to\infty$, and let
$\gamma_{s_i}:[0,d(q,\alpha(s_i))]\to M$ be a minimal geodesic from 
$q$ to $\alpha(s_i)$. Then $\gamma_{s_i}$ converges to $\gamma$ 
as $i\to \infty$.

Because $\gamma$ and $\alpha$ are broken at $x_0$,
by the triangle inequality and similar calculation in $(2.1)$, for 
each $s_i$, we have 
$$F_{p;q}(\alpha(s_i))<F_{p;q}(x_0)=\min 
F_{p;q},$$
and thus for any $0\le t\le d(q,\alpha(s_i))$,
\begin{align*}
F_{p;q}(\gamma_{s_i}(t))\le F_{p;q}(\alpha(s_i))<\min 
F_{p;q},\tag{$\ref{lem:propertyoflift}.3$}
\end{align*}
which implies that none of points in $\gamma_{s_i}$ is a cut point of 
$p$. Therefore there is a unique lift curve $\tilde\gamma_{s_i}$ of 
$\gamma_{s_i}$ starting at $(\exp_p|_{\tilde{D}_p})^{-1}\gamma(0)$ 
in the tangential segment 
domain $\tilde D_p\subset T_pM$.

If $\{\tilde\gamma_{s_i}\}$ is uniformly Lipschitz, then by
the Arzel\`{a}-Ascoli theorem, a subsequence of 
$\{\tilde\gamma_{s_i}\}$ 
converges to a continuous 
curve $\tilde\gamma_\infty:[0,d(q,x_0)]\to T_pM$, which  
satisfies that
$$\tilde\gamma_\infty(0)=(\exp_p|_{\tilde{D}_p})^{-1}\gamma(0), \quad
\tilde\gamma_\infty(d(q,x_0))=w,$$ and
$$\exp_p(\tilde\gamma_\infty(t))=\lim_{i\to\infty}\gamma_{s_i}
(t)=\gamma(t), \text{ for all $0\le t\le d(q,x_0)$.}$$ 
That is, $\gamma$ 
has a unique lifting in the segment domain that satisfies 
$(\ref{lem:propertyoflift}.2)$.

To prove that
$\{\tilde\gamma_{s_i}\}$ is uniformly Lipschitz, it suffices to show 
that
there is $N>0$ such that the
distance between $\bigcup_{i\ge N}
\tilde\gamma_{s_i}([0,d(q,\alpha(s_i))])$ and tangential conjugate 
locus $\tilde J_p\subset T_pM$ is positive. 
Indeed, because $\tilde\gamma_{s_i}(d(q,\alpha(s_i)))$ converges to 
$w=d(p,x_0)\alpha'(0)$, at which the differential $d(\exp_p)$ is 
non-singular, there is 
a small $\delta>0$ and some $\epsilon>0$ such that 
$$d(\tilde\gamma_{s_i}(t),\tilde J_p)>\epsilon,\text{ for all 
$d(q,\alpha(s_i))-\delta\le t\le d(q,\alpha(s_i))$ and large 
$i$}.$$
On the other hand, because the restriction 
$\left.\gamma_{s_i}\right|_{[0,d(q,\alpha(s_i))-\delta]}$ converges 
to $\left.\gamma\right|_{[0,d(q,x_0)-\delta]}$, which lies in the 
segment domain $D_p=\exp_p\tilde D_p\subset M$, there is some 
$\epsilon_1>0$ such that
$$d(\tilde\gamma_{s_i}(t),\tilde J_p)>\epsilon_1,\text{ for all 
$0\le t\le d(q,\alpha(s_i))-\delta$ and large $i$}.$$

Now what remains is to prove the case that the minimal 
geodesic from $q$ to $x_0$ is not unique.
Let us fix some interior point $q_1=\gamma(t_1)$ $(0<t_1<d(q,x_0))$ 
of $\gamma$ and consider the function $F_{p;q_1}:C_p\to \Bbb R$ 
instead.
Because 
\begin{align*}
F_{p;q_1}(x)&=d(q_1,x)+d(x,p)\\
&\ge d(x,q)-d(q,q_1)+d(x,p)\\
&=F_{p;q}(x)-d(q,q_1) \tag{$\ref{lem:propertyoflift}.4$}
\end{align*}
and 
\begin{align*}
F_{p;q_1}(x_0)&=d(x_0,q)-d(q,q_1)+d(x_0,p)\\
&=F_{p;q}(x_0)-d(q,q_1), \tag{$\ref{lem:propertyoflift}.5$}
\end{align*}
we see that $F_{p;q_1}(x)$ also takes minimum at $x_0$. Now the 
minimal 
geodesic from $q_1$ to $x_0$ is unique. By the same argument as 
above,  
either $\left.\gamma\right|_{[t_1,d(q,x_0)]}$ from a whole 
geodesic with $\alpha$ at $x_0$, or it has a unique lift in the 
tangential segment domain. So does $\gamma$.
\end{proof}

A local minimum point of $F_{p;q}$ can be reduced to the case of 
global minimum by the following lemma.
\begin{lemma}\label{lem:localminimum}
Let $x_0$ be a local minimum point of $F_{p;q}:C_p\to \mathbb{R}$ in 
$C_p$, and $\gamma:[0,d(q,x_0)]\to M$ be a minimal geodesic from 
$q=\gamma(0)$ to $x_0=\gamma(d(q,x_0))$. Then for any interior point 
$q_t=\gamma(t)$ of $\gamma$ that is sufficient close to $x_0$, the 
function $F_{p;q_t}:C_p\to \Bbb R$ takes its minimum at $x_0$.
\end{lemma}
\begin{proof}
By $(\ref{lem:propertyoflift}.4)$ and 
$(\ref{lem:propertyoflift}.5)$, $x_0$ is also a local minimum 
point of $F_{p;q_t}$ for any $t\in [0,d(q,x_0))$.
Therefore, it suffices to show that $F_{p;q_t}$ takes its minimum 
near $x_0$ as $q_t$ sufficient close to $x_0$. 

Let us argue by contradiction. Assuming the contrary, one 
is able to find a sequence of points $\{q_i=\gamma(t_i)\}$ that 
converges to $x_0$ such that $F_{p;q_i}$ takes its minimum at some 
point $z_i\in C_p$ outside an open ball of $x_0$,  
$$B_\epsilon(x_0)=\{x\in M\,|\, d(x,x_0)<\epsilon\}.$$
By passing to a subsequence, we assume that $z_i\to 
z_0\in C_p$. Then for any $y\in C_p$,
$$d(p,z_i)+d(z_i,q_i)=F_{p;q_i}(z_i)\le 
F_{p;q_i}(y)=d(p,y)+d(y,q_i).$$
Taking limit of the above inequality, we get
\begin{align*}
F_{p;x_0}(z_0)\le F_{p;x_0}(y), \quad 
\text{ for any $y\in C_p$}.
\end{align*}
Let $y=x_0$, then 
$$d(p,z_0)+d(z_0,x_0)\le F_{p;x_0}(x_0)=d(p,x_0).$$
Because $d(z_0,x_0)\ge \epsilon$, this implies that $z_0$ is an 
interior point of a minimal geodesic from $p$ to $x_0$, and it 
contradicts to the fact that $z_0\in C_p$.
\end{proof}

We now are ready to prove 
Theorem~A.
\begin{proof}[Proof of Theorem~A]
Let $x_0\in C_p$ be a local minimum point of the function 
$$F_{p;q}:C_p\to \Bbb R, 
\quad F_{p;q}(x)=d(x,q)+d(x,p).$$ 
Let $\alpha:[0,d(p,x_0)]\to M$ be
a minimal geodesic from 
$p=\alpha(0)$ to $x_0=\alpha(d(p,x_0))$, along which $p$ is not 
conjugate to $x_0$. Let $\gamma:[0,d(q,x_0)]\to M$ be a 
minimal geodesic from $q=\gamma(0)$ to $x_0=\gamma(d(q,x_0))$, and
$w=d(p,x_0)\alpha'(0)$.

First let us prove that there is a minimal geodesic from $p$ to 
$x_0$ that forms a whole geodesic with $\gamma$. 
According to Lemma~\ref{lem:localminimum},
by moving $q$ to an interior point in $\gamma$ that is sufficient 
close to $x_0$ and denoted also by $q$, it can be reduced to the case 
that $x_0$ is a minimal point of $F_{p;q}$ and $\gamma$ is a 
unique minimal geodesic connecting $q$ and $x_0$. 

By 
Lemma~$\ref{lem:propertyoflift}$, if $\alpha$ and $\gamma$ are broken 
at 
$x_0$, then $\gamma$ has a unique lift $\tilde\gamma:[0,d(q,x_0)]\to 
T_pM$ in the tangential segment domain $\tilde D_p\subset T_pM$, 
whose endpoint satisfies 
$$\tilde\gamma(d(q,x_0))=w.$$
Because $x_0$ is not conjugate to $p$ along $\alpha$, there is 
another minimal geodesic $\beta:[0,d(p,x_0)]\to M$ from $p$ to $x_0$. 
We now prove that $\beta$ must form a whole geodesic with $\gamma$.

Assuming the contrary, that is, $\beta$ does not form a whole 
geodesic with $\gamma$ neither. For any interior point $\beta(s)$  
$(0<s<d(p,x_0))$ in $\beta$, let $l_s=d(q,\beta(s))$ 
and $\gamma_s:[0,l_s]\to M$ be a 
minimal geodesic from $q$ to $\beta(s)$. Then by the same argument 
as $(\ref{lem:propertyoflift}.3)$, for any $0\le t\le l_s$ we have
$$F_{p;q}(\gamma_{s}(t))\le F_{p;q}(\beta(s))<F_{p;q}(x_0)=\min 
F_{p;q},$$
which implies that $\gamma_s$ has a unique lift $\tilde\gamma_s$ in 
the tangential segment domain $\tilde D_p\subset T_pM$.

We point it out that,
because the endpoint $\tilde\gamma(l_s)$ of $\tilde\gamma_s$ may
approach $\tilde J_p\subset T_pM$,
one cannot directly conclude that the family of curves 
$\tilde\gamma_s$ 
contains any convergent subsequence as $s\to d(p,x_0)$.
Instead, let us consider a small metric ball 
$B_\epsilon(w)$ at $w=d(p,x_0)\alpha'(0)$ in $T_pM$ on 
which 
the restriction of $\exp_p$
$$\left.\exp_p\right|_{B_\epsilon(w)}:
B_\epsilon(w)\to M$$
is an diffeomorphism onto its image.
Then $d(p,x_0)\beta'(0)\not\in B_\epsilon(w)$.
Because $\gamma_s$ converges to $\gamma$ as $s\to d(p,x_0)$ and
the restriction $\left.\exp_p\right|_{\tilde D_p}$ of $\exp_p$ on the 
tangential 
segment domain $\tilde D_p$ is a diffeomorphism, 
$\left.\tilde\gamma_s\right|_{[0,l_s)}$ 
converges to $\left.\tilde\gamma_s\right|_{[0,d(q,x_0))}$ 
pointwisely, that is, for $0<t<1$,
$$\tilde\gamma_s(l_s\cdot t)\to 
\tilde\gamma(d(q,x_0)\cdot t)\text{ as $s\to d(p,x_0)$.}$$
 Then
for any $0<t_1<1$ that is sufficient close to $1$, there 
exists $0<s(t_1)<d(p,x_0)$ such that for all $s(t_1)<s<d(p,x_0)$,
$$\tilde\gamma_s(l_st_1)\in 
B_{\frac{\epsilon}{4}}(w),\quad  
\tilde\gamma_s(l_s)\in 
B_{\frac{\epsilon}{2}}(d(p,x_0)\beta'(0)),$$ and 
$\left.\gamma_s\right|_{[l_st_1,l_s]}$ lies in the open 
neighborhood $\exp_p 
B_{\frac{\epsilon}{4}}(w)$ of $x_0$.
Because the lift of 
$\left.\gamma_s\right|_{[l_st_1,l_s]}$ is unique in 
$\tilde D_p$ and $\tilde\gamma_s(l_st_1)\in 
B_{\frac{\epsilon}{4}}(w)$, we conclude that 
$\left.\tilde\gamma_s\right|_{[l_st_1,l_s]}$ lies in 
$B_{\frac{\epsilon}{4}}(w)$, which contradicts to 
the fact that 
$$B_{\frac{\epsilon}{4}}(w)\cap 
B_{\frac{\epsilon}{2}} (d(p,x_0)\beta'(0))=\emptyset.$$

Secondly, we prove that there are at most two geodesics from $q$ to 
$p$ passing through $x_0$. Indeed, if there are two distinct minimal 
geodesics $\gamma_1$, $\gamma_2$ 
connecting $q$ and $x_0$, then one of them, say $\gamma_1$, does not 
form a whole geodesic with $\alpha$. 
By the proof above, 
any other minimal geodesic except $\alpha$ 
connecting $p$ and $x_0$ from a whole geodesic with $\gamma_1$.
By Lemma~\ref{lem:propertyoflift}, $\gamma_2$ form a whole geodesic 
with $\alpha$.
Therefore, there are exactly two minimal geodesics $\alpha$ 
and $\beta$ from $p$ to $x_0$,  the union of $\beta$ and 
$\gamma_1$ form a whole geodesic going through $x_0$, and 
$\alpha$ and $\gamma_2$ form another whole geodesic.
In particular, there won't be a third geodesic 
from $q$ to $x_0$. 

\end{proof}

Theorem~B and Theorem~\ref{thm:newcharacter} are immediate 
corollaries of Theorem~A. We give a proof of 
Theorem~\ref{thm:finiteness} to end the paper.

\begin{proof}[Proof of Theorem~\ref{thm:finiteness}]
Let $\Cal M(n,k,D,v,r)$ be the set consisting of all 
complete $n$-dimensional Riemannian manifolds
whose sectional curvature $\ge k$, diameter $\le D$, 
volume $\ge v$ and essential conjugate radius $\ge r$.
Let $M\in \Cal M(n,k,D,v,r)$. 
By Cheeger's lemma (\cite{Cheeger1970finiteness}, see Theorem~5.8 in 
\cite{CheegerEbin1975comparison}), there is an universal constant 
$c_n(D,v,k)>0$ depending only on $n,D,v,k$ such that
every smooth closed geodesic on $M$ has length $>c_n(D, v, k)$.
Therefore, by Theorem~\ref{thm:injrad} the injectivity radius of $M$ 
is bounded below by $\min\{r, \frac{1}{2}c_n(D,v,k)\}$. The 
finiteness of diffeomorphism classes in the set $\Cal M(n,k,D,v,r)$ 
follows from the standard argument (see \cite{Fukaya1987collapsing, 
Yamaguchi1991collapsing} for example) on the construction of a 
diffeomorphism between two Riemannian manifolds with small 
Gromov-Hausdorff distance .
\end{proof}

\bibliographystyle{plain}
\bibliography{conjugateloops}

\end{document}